\documentclass[10pt,a4paper]{elsarticle}
\usepackage[T1]{fontenc}
\usepackage{amsthm}
\usepackage{amsmath}
\usepackage{bm}
\usepackage{color}
\usepackage{hyperref}
\newtheorem{theorem}{Theorem}
\newtheorem{corollary}{Corollary}[theorem]

\newtheorem{remark}{Remark}

\newcommand{\K}{{\bm K}}
\renewcommand{\P}{{\bm P}}
\renewcommand{\H}{{\bm H}}
\newcommand{\R}{{\bm R}}
\newcommand{\I}{{\bm I}}
\newcommand{\tr}{\text{tr}}
\newcommand{\PK}{\P_\K}
\begin{document}

	\begin{frontmatter}
		
		
		
		{\title{Minimizing effects of the Kalman gain on Posterior covariance  Eigenvalues, the characteristic polynomial and symmetric polynomials of Eigenvalues}}

		
		\author{Johannes Krotz\corref{cor1} \fnref{label1}}
		
		\affiliation[label1]{organization={Department of Mathematics, University of Tennessee Knoxville},
			city={Knoxville},
			postcode={37996}, 
			state={TN},
			country={USA}}

		\begin{abstract}
			The Kalman gain is commonly derived as the minimizer of the trace of theposterior covariance. It is known that it also minimizes the determinant of the posterior covariance. I will show that it also minimizes the smallest Eigenvalue $\lambda_1$ and the chracteristic polynomial on $(-\infty,\lambda_1)$ and is critical point to all symmetric polynomials of the Eigenvalues, minimizing some. This expands the range of uncertainty measures for which the Kalman Filter is optimal.
		\end{abstract}
		
		
	\end{frontmatter}

	In a Kalman filter algorithm the Kalman gain is defined by 
	\begin{align}
		\K^* = \P\H^\top (\H\P\H^\top\R)^{-1}, \label{eq:K}
		\end{align}
		where $\P$ is the covariance matrix of prior, $\R$ is the covariance matrix of the Likelihood and $\H$ is the measurement operator. \cite{asch_data_2016,sarkka2013bayesian}
		The posterior covariance matrix $\P_{\K^*}$ is defined through 
		\begin{align}
			\PK = (\I-\K\H)\P(\I-\K\H)^\top +\K\R\K^\top. \label{eq:PK}
		\end{align} 
		evaluated at $\K^*$, where $\I$ is the identity matrix. Note that as covariance matrices $\P$ and $\R$ are symmetric and striclty positive, properties which $\PK$ inherits.\\
		The Kalman gain defined in equation \eqref{eq:K} is often derived as the minimizer of the total posterior variance, i.e. the trace of $P_K$.\cite{asch_data_2016,sarkka2013bayesian,jazwinski_stochastic_1970,todling_estimation_1999,laue_matrixcalculusorg_2020}. In other words \begin{align}
			\K^* = \arg\min_\K \tr(\PK)
		\end{align}
		It was shown in \cite{bach2021proof} that $\K^*$ also minimizes the posterior generalized variance, defined as the determinant of $\PK$, i.e.
		\begin{align}
		\K^* =	\arg\min_\K \det(\PK). 
		\end{align}
		Let $\Phi(\PK,\lambda) := \det(\lambda \I - \PK) = \prod_{i=1}^n(\lambda - \lambda^\K_i):= \sum_{i=0}^{n} a_i^\K \lambda^i$ be the characteristic Polynomial of $\PK$ and $0<\lambda^\K_1\le\dots\le\lambda^\K_n$ be its roots, the Eigenvalues of $\PK$. Let $\Lambda^\K =\{\lambda^\K_i\}_{i=1}^n$. 
		
		\begin{theorem} \label{thm:1}
			Let $\lambda\notin\Lambda^{\K^*}$ and $g_\lambda(\K):= ||\Phi(\PK,\lambda)||$. The Kalman gain $K^*$ is a critical point of $g_\lambda$ and, if $\lambda<\lambda^{K^*}_1$ \begin{align}
				K^* = \arg\min_\K g_\lambda(K).
			\end{align}
		\end{theorem}
		\begin{proof}
		Let $\lambda \notin \Lambda^{K}$ and let $f(\K,\lambda):= \log g_\lambda$.
		 Using matrix calculus it follows that 
		 \begin{align}
				\frac{df_\lambda(\K)}{\K} = -\left(\lambda \I - \PK\right)^{-1}\frac{d\PK}{d\K}.
			\end{align}
			The equaton $\frac{df_\lambda(\K)}{\K} = 0$ is solved if and only if $\frac{d\PK}{d\K} = 0$ is solved. The latter however is only true for $\K=\K^*$, implying that $K^*$ is a critical point of $f_\lambda$ and hence of $g_\lambda$. \\
			For $\lambda < \lambda^{K^*}_1$ the term $-\left(\lambda \I - \P_{\K^*}\right)^{-1}$ is positive definite, meaning the minimizing properties of $\frac{d\P_{\K^*}}{d\K}$ are preserved. 
		\end{proof}
		\begin{corollary}
			$\K^*$ minimizes the smallest Eigenvalue
			\begin{align}
				\K^* = \arg\min_\K(\lambda_1^{\K}).
			\end{align}
		\end{corollary}
		\begin{proof}
			Let $\K\neq \K^*$. Suppose $\lambda_1^{\K^*}>\lambda_1^{\K}$. Then $||\Phi(\PK),\lambda_1^\K||=0<||\Phi(\P_{K^*}),\lambda_1^\K||$. By continuity of these characteristic Polynomials there is $\lambda<\lambda_1^\K<\lambda_1^{K^*}$, i.e. $\lambda\notin \Lambda^\K \cup \Lambda^{\K^*}$ such that $||\Phi(\PK),\lambda||<||\Phi(\P_{K^*}),\lambda||$. This contradicts Theorem \ref{thm:1}.
		\end{proof}
		\begin{corollary}
			Let $\Phi(\PK,\lambda) := \sum_{i=0}^n a_i^\K \lambda^i$. Then $\K^*$ is critical point of the map $\K \mapsto ||a_i||^\K$ for all $i=1,\dots,n$ and minimizer for even $i$:\begin{align}
				\K^* = \arg\min_\K ||a_i^\K|| & & \text{for} \mod(i,2)=0.
			\end{align}
		\end{corollary}
		\begin{proof}
			For $a_0$ the claim holds by applying Theorem \ref{thm:1} at $\lambda=0$.
			For $j\in \{\dots,n\}$ let $\Psi_j(\lambda) = \sum_{j \neq i}^{n}  a_i^K \lambda^i$. Pick $\mu_j$ such that $\Psi_j(\mu_j)=0$. Then let $\phi_{j,\lambda}(\K):=\frac{|| \Phi(\PK,\lambda)||}{||\mu_j||^j}$. This function $\phi$ has the same critical points and extrema as $g_\lambda$ in Theorem \ref*{thm:1}.  Since $\phi_{j,\mu_j}(\K)=||a_j^\K||$ the first part of the claim follows.
			
			 The second part follows if $\mu_j <\lambda_1^{\K^*}$ for even all $j$.
			Hence let $j\neq 0$ now be an even number. Since all Eigenvalues of $\PK$ are strictly positive and the $a_i^\K$ results from expanding the product $\prod_{i=1}^n(\lambda-\lambda_i^K)$, he signs of the $a_i^\K$ are alternating, meaning $sgn(a_i^\K)= -sgn(a_{i+1}^\K)$ for all $i=0,\dots,n-1$. Since we ultimately care about $|| \Phi(\PK,\lambda)||$ I will assume WLOG that $a_i>0$ for even $i$. Hence $\Psi_j(0)=a_0 = \Phi(\PK,0)>0$ and $\Psi_j(\lambda)=\Phi(\PK,\lambda)-a_j^\K\lambda^j <\Phi(\PK,\lambda)$. Therefore $\Psi_j(\lambda)$ has a root, which we chose to be $\mu_j$, between $0$ and $\lambda_1^\K<\lambda_1^{\K^*}$.
		\end{proof}
		\begin{remark}[Elementary symmetric Polynomials]
			The elementary symmetric Polynomials in $n$ variables $e_1(X_1,\dots,X_n),\dots, e_n(X_1,\dots,X_n)$ are defined as \begin{align}
				e_k(X_1,\dots,X_n) = \sum_{1\le j_1 < \dots < j_k \le n } X_{j_1}\cdots X_{j_k}.
			\end{align}
			Examples are $e_1(X_1,\dots,X_n) = X_1 + \dots X_n$ and $e_n(X_1,\dots,X_n) = X_1\cdots X_n$. They are invariant under permutation of their entries and appear naturally as the coefficients of the characteristic Polynomial: \begin{align}
				\prod_{i=1}^n (\lambda-\lambda_i) = \lambda^n-e_1(\lambda_1,\dots,\lambda_n)\lambda^{n-1}+\dots +(-1)^n e_n(\lambda_1,\dots,\lambda_n). 
			\end{align}
			The previous theorem thus also applies to these elementary symmetrical polynomials evaluated at Eigenvalues of $\PK$.
		\end{remark}
		\begin{corollary}
			The Kalman gain $\K^*$ is a critical point of the map $\K\mapsto Q(\lambda_1^\K,\dots,\lambda_n^\K)$, where $Q$ is an arbirtray symmetric Polynomial, i.e. $Q(X_1,\dots,X_n)=Q(X_{\sigma_1},\dots, X_{\sigma_n})$, where $\sigma \in S_n$ is a permutation.
		\end{corollary}
		\begin{proof}
			The non-leading coefficients $a^\K_0,\dots,a^\K_{n-1}$ of the characteristic polynomial $\Phi(\P_K,\lambda)$ are the elementary symmetric polynonials $e_1,\dots,e_n$ evaluated at the Eigenvalues of $\PK$. 
			From the previous corollary we showed that $K^*$ is a critical point for these. By the fundamental theorem of symmetric polynomials there is a polynomial $P(X_1,\dots,X_n)$ such that 
			\begin{align}
Q(\lambda_1^\K,\dots,\lambda_n^\K) = P(e_1(\lambda_1^\K,\dots,\lambda_n^\K),\dots,e_n(\lambda_1^\K,\dots,\lambda_n^\K)).
			\end{align}
			The claim follows by chain rule.
		\end{proof}
		\begin{remark}[Special cases]
			For $\lambda<0$ we can see that $||\Phi(\PK,\lambda)|| = \sum_{i=0}^n |a^\K_i|||\lambda||^i$. The Kalman gain $K^*$ minimizes this function for all $\lambda<\lambda^K_1$. For $\lambda = -1$ it is found that the sum of all $a_i^K$, i.e. the sum of all elementary symmetric Polynomials in Eigenvalues of $\PK$ are minimized by $\K^*$.
			At $\lambda=0$ this evaluates to $a^{\K^*}_0 = \det (\P_{K^*})$ reproducing the result from  \cite{bach2021proof} .
			It is easy to see that $\frac{||||\Phi(\PK,\lambda)||-||\lambda||^n||}{||\lambda||^{n-1}}\rightarrow a^{\K}_{n-1}=\tr (\PK)$ as $\lambda \rightarrow -\infty$. Since $\K^*$ minimizes this along the limiting process the minimizing of $\tr(\PK)$ is rediscovered. 
		\end{remark}

		\bibliography{kalman_gain_proof}{}
		\bibliographystyle{plain}
\end{document}